\documentclass[11pt,twoside]{article}
\usepackage{times}
\usepackage{CJK}
\usepackage{mathptmx}
\usepackage{amsmath,amssymb,amsthm}
\usepackage{enumerate}
\usepackage{cite}
\usepackage{mathrsfs,graphicx}
\usepackage{float}
\usepackage{stmaryrd}
\usepackage[margin=1in]{geometry}  
\usepackage{graphicx}
\usepackage{amssymb}             
\usepackage{amsmath}
\usepackage{slashed}            
\usepackage{amsfonts}              
\usepackage{amsthm}
\usepackage{breqn}
\usepackage{verbatim}              
\usepackage{enumitem}
\usepackage{fancyhdr}
\usepackage{color}
\usepackage[pagewise]{lineno}
\usepackage{amssymb}
\usepackage{slashed}
\usepackage{bbm}

\usepackage{ifpdf}
\usepackage{color}
\usepackage{graphicx}
\DeclareGraphicsExtensions{.eps,.bmp}
\begin{document}

\footskip=0pt
\footnotesep=2pt

\allowdisplaybreaks

\newtheorem{claim}{Claim}[section]

\theoremstyle{definition}
\newtheorem{thm}{Theorem}[section]
\newtheorem*{thmm}{Theorem}
\newtheorem{mydef}{Definition}[section]
\newtheorem{lem}[thm]{Lemma}
\newtheorem{prop}[thm]{Proposition}
\newtheorem{remark}[thm]{Remark}
\newtheorem*{propp}{Proposition}
\newtheorem{cor}[thm]{Corollary}
\newtheorem{conj}[thm]{Conjecture}
\def\F{\mathcal{F}}
\def\la{\lambda}
\def\tR{\tilde{R}}
\def\ds{\displaystyle}
\def\ve{\varepsilon}
\def\al{\alpha}
\def\be{\beta}
\def\na{\nabla}
\def\dl{\delta}
\def\ls{\lesssim}
\def\dl{\delta}
\def\de{\delta}
\def\om{\omega}
\def\ve0{\varepsilon_0}
\def\ve{\varepsilon}
\def\De{\Delta}
\def\vp{\varphi}
\def\t{\tilde}
\newcommand{\bd}[1]{\mathbf{#1}}  
\newcommand{\R}{\mathbb{R}}      
\newcommand{\ZZ}{\mathbb{Z}}      
\newcommand{\Om}{\Omega}
\newcommand{\bv}{\mathbf{v}}
\newcommand{\bn}{\mathbf{n}}
\newcommand{\bU}{\mathbf{U}}
\newcommand{\q}{\quad}
\newcommand{\p}{\partial}
\newcommand{\n}{\nabla}
\newcommand{\f}{\frac}

\numberwithin{equation}{section}

\title{The critical power of short pulse initial data on
the global existence or blowup of smooth solutions to
3-D semilinear Klein-Gordon equations}
\author{}

\author{Shen Jindou$^1$, \quad  Yin Huicheng$^{2,}$
\footnote{Shen Jindou (jdshenmath@163.com) and  Yin Huicheng
(huicheng@nju.edu.cn, 05407@njnu.edu.cn) are
supported by the NSFC (No.12331007) and the National key research and development program (No.2020YFA0713803).
}
\vspace{0.3cm}\\
\small
1. Department of Mathematics, Nanjing University, Nanjing 210093, China.\\
\small
2. School of Mathematical Sciences and Mathematical Institute,\\
\small
Nanjing Normal University, Nanjing 210023, China.\\}
\vspace{0.2cm}

\date{}
\maketitle
\centerline{}

\date{}
\maketitle
\thispagestyle{empty}
\begin{abstract}
\vspace{0.2cm}
It is well-known that there are global small data smooth solutions
for the 3-D semilinear Klein-Gordon equations
$\square u + u = F(u,{\partial u})$  with cubic nonlinearities.
However, for the short pulse initial data
$(u, \p_tu)(0, x)=({\delta^{\nu+1}}{u_0}({\frac{x}{\delta}}),{\delta^\nu }{u_1}({\frac{x}{\delta}}))$
with $\nu\in\Bbb R$ and $(u_0, u_1)\in C_0^{\infty}(\Bbb R)$,
which are a class of large initial data,
we establish that when $\nu\le -\f{1}{2}$, the solution $u$ can blow up in finite time for some suitable
choices of $(u_0, u_1)$ and cubic nonlinearity $F(u,{\p u})$;
when  $\nu>-\f{1}{2}$, the smooth solution $u$ exists globally.
Therefore, $\nu=-\f{1}{2}$ is just the critical power corresponding to the global existence or blowup of smooth
 short pulse solutions for the  cubic semilinear Klein-Gordon equations.
\end{abstract}

\vskip 0.2cm


\vskip 0.2 true cm


\vskip 0.2 true cm

\section{Introduction}

In this paper, we are concerned with the following 3-D  cubic semilinear Klein-Gordon equation

\begin{equation}\label{equation}
\left\{ \begin{array}{l}
\square u + u = F(u,\partial u),\quad (t,x)\in \Bbb R^{1+3},\\
u(0, x) = {\delta ^{\nu + 1}}{u_0}({\frac{x}{\delta }}),\quad {\partial _t}u(0,x) ={\delta^\nu }{u_1}(\frac{x}{\delta}),
\end{array} \right.
\end{equation}
where $t\ge 0$, $x = ({x_1}, x_2 ,{x_3})\in\Bbb R^3$,  $\partial=(\partial_{0},\partial_{1},\partial_{2},\partial_{3})$
with $\partial_{0}=\partial_{t}$ and $\partial_{a}=
\partial_{x_{a}}$ $(a=1,2,3)$, $\square  = \partial _0^2 - \Delta$ with $\Delta= \sum\limits_{a = 1}^3 {\partial _a^2}$,
$\dl>0$ is sufficiently small, $\nu\in\Bbb R$ and $(u_0,u_1)\in C_0^\infty(\Bbb R^3)$.
In addition, $\partial_{-1}=\mathbbm{1}$ and
$F(u,\partial u) = \sum\limits_{j,k,l =  - 1}^3 {g^{jkl}{\partial _j}u{\partial _k}u{\partial _l}u}$
with $g^{jkl}\in \mathbb{R} (j,k,l=-1,0,\cdots,3)$ being some constants.
Without loss of generality, $\text{supp}u_0 ,\text{supp}u_1\subset \left\{x:|x|\le 1\right\}\text{and}~{\left\| {{u_0}} \right\|_{{H^{N + 1}(\mathbb{R}^3)}}} + {\left\| {{u_1}} \right\|_{{H^{N}(\mathbb{R}^3)}}} \le 1$ $(N\in \Bbb N,~N\ge 4)$ are assumed.

Set $\partial_\delta:=\delta( \partial_t, \partial_1, \partial_2, \partial_3)$,
$L_\delta:=\delta (L_1, L_2, L_3)$
with $L_a=x_a\partial_t+t\partial_a$ ($a=1,2,3$) and $L=\left(L_1, L_2, L_3\right)$.

Our main results are
\begin{thm}\label{mainthm}
For $\nu>-\f12$ and small $\delta>0$, problem \eqref{equation} admits a global smooth solution
$u \in C([0,\infty),$ $H^{N + 1}(\mathbb{R}^3))\cap {C^1}([0,\infty),{H^N}(\mathbb{R}^3))$ such that
for $|I|+|J| \le N - 2$ and $t \ge 0$,
\begin{equation}\label{0-S}
\begin{aligned}
\left| {{\partial_\delta ^{\rm{I}}}{L_\delta^J}u\left( {t,x} \right)} \right| &\ls  \left\{ \begin{array}{l}
{{\delta ^{\nu+1} }  },\qquad \qquad \quad ~t\le \delta/2,
\\
{{\delta ^\nu } {{\left(\delta^{-1} t+2\right)}^{ - \frac{3}{2}}}},~~~t>\delta/2,
\end{array} \right.
\\
\left| {\partial_\delta{\partial_\delta ^{\rm{I}}}{L_\delta^J}u\left( {t,x} \right)} \right| &\ls  \left\{ \begin{array}{l}
{{\delta ^{\nu+1} }  },\qquad \qquad \quad ~~~t\le \delta/2,
\\
{{\delta ^{\nu+1} } {{\left(\delta^{-1} t+2\right)}^{ - 1}}},~~t>\delta/2.
\end{array} \right.
\end{aligned}
\end{equation}
\end{thm}

\begin{thm}\label{mainthm-1}
For $\nu\le -\f12$ and $F(u,\partial u)=(\partial_t u)^2\partial_1u$,
there exists a class of $(u_0(x), u_1(x))$ such that the solution $u$ of  \eqref{equation} can blow up  in finite time.
\end{thm}

\begin{remark}\label{R0-1}
From Theorems \ref{mainthm} and \ref{mainthm-1}, one knows that
$\nu=-\f{1}{2}$ is just the critical power corresponding to the global existence or blowup of smooth
 short pulse solutions to problem \eqref{equation} in the case of  cubic semilinear Klein-Gordon equation.
\end{remark}

Let us recall some basic results on the global small data smooth solutions to the nonlinear Klein-Gordon equations
\begin{equation}\label{0-E}\begin{cases}
\Box u+u=F(u,\partial u),\quad (t,x)\in \Bbb R^{1+n},\\[2mm]
(u,\p_t{u})(0,x)=\dl (u_{0},u_{1})(x),
\end{cases}
\end{equation}
where $t\ge 0$, $x=(x_1, ..., x_n)$, $n\ge 1(n\in \Bbb N)$, $\dl>0$ is sufficiently small, $(u_{0},u_{1})\in C_0^{\infty}(\Bbb R^n)$,
$F$ is at least quadratic with respect to its
arguments.  It is well-known that when $n\ge 2$,
\eqref{0-E} admits a global smooth solution $u$
(see  \cite{K}, \cite{OTT}, \cite{S} and \cite{ST}); when $n=1$ and $F$ satisfies the corresponding null
condition, \eqref{0-E} has a global solution (see \cite{De-01}).
Therefore, the cubic semilinear
equation in \eqref{equation} has a global small data solution.
In addition, for the small initial data with  noncompact supports, there also are many
interesting  global existence or long time existence results on the nonlinear
Klein-Gordon equations (see \cite{DF00}, \cite{HY23}-\cite{HY23-1} and \cite{Stingo18}-\cite{Stingo23}).

On the other hand, consider the following 3-D quasilinear wave equation
 \begin{equation}\label{main equation}
  -\big(1+(\p_tu)^p\big)\p_t^2u+\De u=0,\quad p\in\mathbb N
 \end{equation}
together with  such a class of short pulse initial data
 \begin{align}\label{initial data}
  (u,\p_tu)(1,x)=\big(\de^{2-\ve_0}u_0(\frac{r-1}{\de},\om),
    \de^{1-\ve_0}u_1(\frac{r-1}{\de},\om)\big),
 \end{align}
where $0<\ve_0<1$, $\om=\f{x}{r}\in\Bbb S^2$ with $r=\sqrt{x_1^2+x_2^2+x_3^2}$ and $(u_0(\rho,\om), u_1(\rho,\om))\in C_0^{\infty}\left((-1,0)\times\mathbb{S}^2\right)$.
Under the outgoing constraint conditions
 \begin{equation}\label{condition on data}
   (\p_t+\p_r)^ku(1,x)=O(\de^{2-\ve_0-k\max\{0,1-(1-\ve_0)p\}}),\ k=1,2,
   \end{equation}
it is shown that when $p>p_c$ with $p_c=\f{1}{1-\ve_0}$, \eqref{main equation} with \eqref{initial data}
and \eqref{condition on data} has a global smooth solution $u$; when $1\le p\le p_c$, the smooth
solution $u$ can blow up and further the shock is formed for the suitable choice of $(u_0,u_1)$
(see \cite{Ding6}, \cite{Lu1} and \cite{MY}).
Here it is pointed out that the short pulse initial data \eqref{initial data} are firstly introduced
by D. Christodoulou in \cite{C09}. By the
short pulse method, the authors in monumental papers \cite{C09} and \cite{K-R}
showed that the black holes
can be formed in vacuum spacetime and the blowup mechanism is due to the condensation of the
gravitational waves for the 3-D Einstein general relativity equations.

Motivated by the results in  \cite{Ding6}, \cite{Ding-Xin-Yin} and \cite{Lu1}, we start to study the global existence
or blowup of smooth
 short pulse solutions to problem  \eqref{equation}  for variant scope of $\nu$, which is different from the global
 existence problem of small data solutions mentioned above.
By  taking the transformation $(\tau, z)=( \frac{t}{\delta }+2,\frac{x}{\delta})$ and
writing $w(\tau,z):=u\left(\dl(\tau-2),\dl z\right)$, then $\eqref{equation}$  can be changed into the
following form (without confusions,
$(\tau,z)$ is still denoted by $(t,x)$ and $w(\tau,z)$ is denoted by $u(t,x)$ accordingly)
\begin{equation}\label{scaling equation}
\left\{ \begin{array}{l}
\square u + \delta^2u = {F_\delta}(u,\partial u),\quad t\ge 2, x\in\Bbb R^3,\\
u(2,x)= {\delta ^{\nu + 1}}{u_0}(x),\quad {\partial _t}u(2,x)={\delta^{\nu +1} }{u_1}(x),
\end{array} \right.
\end{equation}
where ${F_\delta }(u,\partial u)= \sum\limits_{j,k,l =  - 1}^3 {\t g}^{jkl}(\dl){\partial _j}u{\partial _k}u{\partial _l}u$
and ${\t g}^{jkl}(\dl) = g^{jkl} \delta^{2-n_{j}-n_{k}-n_{l}}$
with $n _{-1}=0$ and $n _{i}=1$ for $i=0,1,2,3$.

In order to investigate \eqref{equation}, it suffices to study problem \eqref{scaling equation}.
For this purpose, as in \cite{K}, \cite{H} and \cite{L-Ma}, we will focus on
the hyperboloidal foliation $\varkappa_s := \{(t,x)\in  \mathcal{K}: s^2 =t^2-|x|^2\}$ with
$s\ge s_0^*=\sqrt{3}$ and $\mathcal{K}=~\{(t,x) : t\geq t_0=2, t\geq |x|+1\}$ for problem \eqref{scaling equation}.
Define the weighted energy $\mathcal{E}^\delta(s,u)$ of \eqref{scaling equation} as
\begin{equation}\label{energy}
\begin{aligned}
\mathcal{E}^\delta (s,u) &:=\int_{\varkappa_s} {[(\partial_t u)^2+\sum_{a=1}^3(\partial_a u)^2+2\sum_{a=1}^3(\frac{x^a}{t}\partial
_t u\partial_a u) +\delta^2u^2]dx}\\
\end{aligned}
\end{equation}
and
\begin{equation}\label{2-15}
M_N^\delta(s):=\sum_{|I|+|J|\leq N}\mathcal{E}^\delta(s,\partial^IL^Ju)^{\frac{1}{2}}.
\end{equation}
Under the assumption $M_N^\delta(s) \le Q\delta^{\nu +1}$  for $s\in [s_0, \infty)(s_0=2)$ and some suitable positive constant $Q>0$,
when $M_N^\delta(s) \le \f{Q}{2}\delta^{\nu +1}$ is shown by the energy estimate,  the smallness of $\dl>0$
and the fact $\nu>-\f12$,
Theorem \ref{mainthm} can be derived by the continuity argument. To prove Theorem \ref{mainthm-1},
motivated by Proposition 7.8.8 of \cite{H} for 1-D semilinear Klein-Gordon equation,
we define the functional for problem \eqref{equation} as follows
\begin{equation}\label{5-5}
\Upsilon (t)=\int_{{\mathbb{R}^3}} {{\partial _t}u{\partial _{{1}}}udx}.
\end{equation}
By the equation $\Box u+u=(\partial_t u)^2\partial_1u$ and integration by parts, one can obtain
\begin{equation}\label{HC-1}
\Upsilon^2(t)\le(t+\delta)^3\Upsilon'(t).
\end{equation}
For some suitable choices of $(u_0,u_1)$, $\Upsilon(0)>18\delta^2$ holds under the condition of $\nu\le-\f12$.
This, together with \eqref{HC-1}, yields the blowup result in Theorem \ref{mainthm-1}.

\vskip 0.2 true cm

\noindent \textbf{Notations:} Through the paper,
we use $a,b,c,  \in \{1,2,3\} $ and $\alpha, \beta,\gamma,\cdots\in \{0,1,2,3\}$ to denote the
space indices and the spacetime indices, respectively.  Both $\dl_{ab}$ and $\dl_{\alpha\beta}$ are Kronecker symbols.
In addition, $f_1\lesssim f_2$ means that there exists a
generic positive constant
$C$ independent of the parameter $\dl>0$ in \eqref{equation} and the variables $(t,x)$ such that $f_1\leq Cf_2$.

In addition, the following notations will be used:
\begin{align*}
&t_0:=2,\quad s_0:=2,\quad s_0^*:=\sqrt{3},\\
&\Omega_{ab}=x_a\partial_b-x_b\partial_a,\quad a,b=1,2,3,~a<b,\\
&L_a=x_a\partial_t + t\partial_a,\quad \underline{\partial}_a=\frac{x_a}{t}\partial_t + \partial_a,
\quad a=1,2,3,\\
&\p=\left(\p_0,\p_1,\p_2,\p_3\right),\quad \p_x=\left(\p_1,\p_2,\p_3\right),\quad L=\left(L_1,L_2,L_3\right),
\\
&\underline{\partial}_\perp =  \partial_t + \sum_{a=1}^3{\frac{x_a}{t}\partial_a},
\\
&\mathcal{K}_{\left[s_0,\infty\right)} \colon=~\{ (t,x) :~ s_0^2 \le t^2-|x|^2< \infty\},
\\
&\mathcal{K} \colon=~\{(t,x) : t\geq t_0, ~t\geq |x|+1\},
\\
&\varkappa_s := \{(t,x)\in \mathcal{K}:~ s^2 =t^2-|x|^2\}\quad\text{for $s\ge s_0^*$},\\
&G_s:=\{(t,x)\in \mathcal{K}:~s_0^2 \le t^2-|x|^2\leq s^2\}\quad\text{for $s\ge s_0$},\\
&V_s:=\{(t,x)\in \mathcal{K}:~(s_0^*)^2 \le t^2-|x|^2\leq s^2\}\quad\text{for $s_0\ge s\ge s_0^*$},\\
&D_s:=\{(t,x): t=2,~ 4-s^2 \le |x|^2\leq 1\}\quad\text{for $s_0\ge s\ge s_0^*$},
\\
&\Lambda_s:=\{(t,x): t=2,~ |x|^2=4-s^2\}\quad\text{for $s_0\ge s\ge s_0^*$},
\\
&\left\| u \right\|_{L_{\varkappa_s}^2}^2
:=\int_{\varkappa_s}|u(t,x)|^2dx=\int_{\Bbb {R}^{3}} \left|u(\sqrt{s^2+|x|^2
},x)\right|^2dx,
\\
&\left\| u \right\|_{L_{\varkappa_s}^\infty}:=
\sup_{\varkappa_s} \left|u(t,x)\right|.
\end{align*}
\input{tu.TpX}

\section{Preliminaries}

Note that the energy $\mathcal{E}^\delta$ in \eqref{energy} is equivalent to
\begin{equation}\label{equivalent energy}
\begin{aligned}
\mathcal{E}^\delta (s,u)&=\int_{\varkappa_s}[(\frac{s}{t}\partial_t u)^2+\sum_{a=1}^3 (\underline{\partial}_a u)^2 + \delta^2u^2]dx \\
&=\int_{\varkappa_s}[(\underline{\partial}_\perp u)^2 + \sum_{a=1}^3 (\frac{s}{t}\partial_a u)^2 + \sum_{1\le a<b\le 3}(t^{-1}\Omega_{ab} u)^2 +\delta^2u^2]dx.
\end{aligned}
\end{equation}
For $s\in[s_0^*,\infty)$ and $|I|+|J|\leq N$,
it follows from \eqref{equivalent energy} that
\begin{equation}\label{Y0-4}
\sum_{\alpha=0}^3\|\frac{s}{t}\partial_\alpha \partial^IL^Ju\|_{L_{\varkappa_s}^2}+\delta\| \partial^IL^Ju\|_{L_{\varkappa_s}^2} \lesssim \mathcal{E}^\delta (s,\partial^IL^Ju)^{1/2}.
\end{equation}

\begin{lem}\label{Y-01}
(Proposition 5.1.1 in \cite{L-Ma})
Let $\phi$ be a sufficiently smooth function supported in  $\mathcal{K}_{\left[s_0,\infty\right)}$.
Then for $s=\sqrt{t^2-|x|^2}\ge s_0$, one has
\begin{equation}\label{Y-Sob}
\sup_{\varkappa_s}t^{3/2}|\phi(t,x)|\lesssim \sum_{|I|\leq 2}\|
L^I\phi\|_{L_{\varkappa_s}^2}.
\end{equation}
\end{lem}

\begin{lem}\label{Y-02}
For any sufficiently smooth function $\phi$ supported in  $\mathcal{K}$ and $|I|\le N$, it holds
\begin{equation} \label{commutation}
\begin{aligned}
L^I(\partial_\alpha \phi)&=\sum_{|I^'|\leq |I|}\sum_{\alpha'=0}^{3}c_{\alpha\alpha' }^{I'I}\partial_{\alpha'}L^{I'}\phi,
\end{aligned}
\end{equation}
where $c_{\alpha\alpha' }^{I'I} \in \mathbb{Z}$ are some suitable constants.
\end{lem}
\begin{proof}
Note that
\begin{equation}\label{Y-03}
\begin{aligned}
\left[ {\partial_0,L_a} \right]&=\partial_a,\quad 1\leq a\leq 3,\\
[\partial_b,L_a]&=\delta_{ab}\partial_0,\quad 1\leq a,b\leq 3.
\end{aligned}
\end{equation}
Then \eqref{commutation} comes from \eqref{Y-03} and direct computations.
\end{proof}

\begin{lem}\label{Y-04}
For any sufficiently smooth function $\phi$ supported in  $\mathcal{K}$, then for $s=\sqrt{t^2-|x|^2}\ge s_0,~ |I|+|J|\leq N-2,$
it holds that
\begin{equation} \label{Y-05}
\sum_{\alpha=0}^3\|\frac{s}{t}\partial_\alpha \partial^IL^J\phi\|_{L_{\varkappa_s}^\infty}+\delta\| \partial^IL^J\phi\|_{L_{\varkappa_s}^\infty} \lesssim s^{-\frac{3}{2}}\sum_{|I'|+|J'|\leq N}\mathcal{E}^\delta (s,\partial^{I'}L^{J'}\phi)^{1/2}.
\end{equation}
\end{lem}
\begin{proof}
It follows from Lemma \ref{Y-01} that
\begin{equation}\label{CY-1}
\sup_{\varkappa_s}t^{3/2}|\frac{s}{t}\partial_{\alpha} \partial^IL^J\phi|\lesssim \sum_{|K|\leq 2}\|
L^K\left(\frac{s}{t}\partial_{\alpha} \partial^IL^J\phi\right)\|_{L_{\varkappa_s}^2}.
\end{equation}
Note that
\begin{equation*}
L^K\left(\frac{s}{t}\partial_{\alpha}  \partial^IL^J\phi\right)=\sum\limits_{K_1+K_2=K} {L^{K_1}\left(\frac{s}{t}\right)L^{K_2}(\partial_{\alpha}  \partial^IL^J\phi)}
\end{equation*}
and
\begin{equation*}
L_a\left(\frac{s}{t}\right)=-\frac{x_a}{t}\frac{s}{t},\quad
L_bL_a\left(\frac{s}{t}\right)=2\frac{x_ax_b}{t^2}\frac{s}{t}-\delta_{ab}\left(\frac{s}{t}\right),\quad a,b=1,2,3.
\end{equation*}
Due to $\left|\frac{x_a}{t}\right|\leq 1$ by $|x|+1\leq t,$
then
$$\left|L^{K_1}\left(\frac{s}{t}\right)\right|\lesssim \frac{s}{t}.$$
Together with \eqref{Y-03}, this yields
\begin{equation}\label{CY-2}
\sum_{|K|\leq 2}\|L^K\left(\frac{s}{t}\partial_{\alpha}  \partial^IL^J\phi\right)\|_{L_{\varkappa_s}^2}\lesssim \sum_{|K_2|\leq 2}\|\frac{s}{t}L^{K_2}\left(\partial_{\alpha}  \partial^IL^J\phi\right)\|_{L_{\varkappa_s}^2}\lesssim \sum_{|I'|+|J'|\leq N}\sum_{\beta=0}^3\|\frac{s}{t}
\partial_{\beta}  \partial^{I'}L^{J'}\phi\|_{L_{\varkappa_s}^2}.
\end{equation}
On the other hand, similarly to the derivations of \eqref{CY-1} and \eqref{CY-2}, we have
\begin{equation}\label{CY-3}
\sup_{\varkappa_s}t^{3/2}| \partial^IL^J\phi|\lesssim \sum_{|K|\leq 2}\|
L^K\left( \partial^IL^J\phi\right)\|_{L_{\varkappa_s}^2}\lesssim \sum_{|I'|+|J'|\leq N}\| \partial^{I'}L^{J'}\phi\|_{L_{\varkappa_s}^2}.
\end{equation}
Due to $s\le t\le \frac{s^2+1}{2}$
for  $(t,x)\in \varkappa_s$ and $s\ge s_0$, then \eqref{Y-05} is shown by \eqref{CY-1}-\eqref{CY-3}.
\end{proof}

\begin{lem}\label{Y-06}
For any sufficiently smooth function $\phi$ supported in  $\mathcal{K}$, we have that for $s=\sqrt{t^2-|x|^2}\ge s_0$, $|I|+|J|\leq N-2$,
\begin{equation}\label{Y-07}
\sum_{\alpha=0}^3\|\frac{t}{s}\partial_\alpha \partial^IL^J\phi\|_{L_{\varkappa_s}^\infty}\lesssim s^{-1} \sum_{|I'|+|J'|\leq N} \mathcal{E}^\delta (s,\partial^{I'}L^{J'}\phi)^{1/2}.
\end{equation}

\end{lem}
\begin{proof}
For $(t,x)\in\varkappa_s,$  it follows from  Lemma \ref{Y-01} and \eqref{CY-2} that
\begin{equation*}
\begin{split}
\frac{t}{s}t^{3/2}|\partial_\alpha  \partial^IL^J\phi(t,x)|
&=\left(\frac{t}{s}\right)^2t^{3/2}|\frac{s}{t}\partial_\alpha  \partial^IL^J\phi(t,x)|
\lesssim \left(\frac{t}{s}\right)^2\sum_{|K|\leq 2}\|
L^K\left( \frac{s}{t} \partial_\alpha  \partial^IL^J\phi\right)\|_{L_{\varkappa_s}^2}
\\
&\lesssim \left(\frac{t}{s}\right)^2\sum_{|I'|+|J'|\leq N}\sum_{\beta=0}^3\|\frac{s}{t}
\partial_\beta  \partial^{I'}L^{J'}\phi\|_{L_{\varkappa_s}^2}.
\end{split}
\end{equation*}
This yields
\begin{equation*}
\begin{aligned}
\frac{t}{s}|\partial_\alpha  \partial^IL^J\phi(t,x)|&\lesssim \frac{t^{\frac{1}{2}}}{s^2}\sum_{|I'|+|J'|\leq N}\sum_{\beta=0}^3\|\frac{s}{t}
\partial_\beta  \partial^{I'}L^{J'}\phi\|_{L_{\varkappa_s}^2}
\lesssim s^{-1}\sum_{|I'|+|J'|\leq N}\sum_{\beta=0}^3\|\frac{s}{t}
\partial_\beta  \partial^{I'}L^{J'}\phi\|_{L_{\varkappa_s}^2}.
\end{aligned}
\end{equation*}
Then \eqref{Y-07} is shown.
\end{proof}

\section{Local existence}
At first, we establish the local existence of solution $u$ to \eqref{scaling equation}
in $[2,T_0]\times\mathbb{R}^3$ ($T_0>5/2$ is a large fixed constant) and derive the smallness orders
of $\dl$
for the derivatives of $u$.
Although the proof procedure is standard, due to the appearance of  small factor $\dl$ and the restriction
of $\nu>-\f12$, we still give the details for reader's convenience.

\begin{prop}
 For $\nu>-1/2$ ,  problem \eqref{scaling equation} has a local solution
$u\in C([2,T_0], H^{N+1}) \cap C^1([2,T_0], H^{N} )$
and $u$ satisfies
\begin{equation}\label{L-S}
\|\partial^Iu(t,x)\|_{L^\infty}\leq C\delta^{\nu +1},
\end{equation}
where $|I|\leq N-1$, and $C>0$ is a generic constant.
\end{prop}

\begin{proof}
Set $u^{(-1)}\equiv 0$.  Define $u^{(m)}$ ($m=0,1,...,$) inductively by
\begin{equation}\label{L-E}
\begin{aligned}
\left\{ \begin{array}{l}
\Box u^{(m)}=f_\delta(u^{(m-1)},\partial u^{(m-1)}),\\
u^{(m)}\left( {2,x} \right) = {\delta ^{\nu  + 1}}{u_0}\left( x \right),{\rm{ }}{\partial _t}u^{(m)}
\left( {2,x} \right)
= {\delta ^{\nu +1} }{u_1}\left( x \right),
\end{array} \right.
\end{aligned}
\end{equation}
where $f_\delta(u,\partial u)=-\delta^2u+F_\delta(u,\partial u)$.

At first, we claim that for small $\delta>0$, there exists a fixed constant $A>0$ such that
for $m\in \mathbb{N} \cup \left\{ { - 1} \right\}$ and $t\in [2, T_0]$,
\begin{equation} \label{L-A}
\begin{aligned}
A_{m}(t):=(\| u^{(m)}(t,x)\|_{H^{N+1}}+\| \partial u^{(m)}(t,x)\|_{H^{N}})\leq A\delta^{\nu +1},
\end{aligned}
\end{equation}
where $A$ is at least double the corresponding bound which is derived from $(u^{(0)}, \p u^{(0)})(2,x)$.

Obviously, $u^{(-1)}$ satisfies \eqref{L-A}. Assume $A_{m-1}(t)\leq A\delta^{\nu +1}$ for all $t\in [2,T_0]$. Now we show that $A_{m}(t)$ fulfills \eqref{L-A} for small $\dl$ and $t\in [2,T_0]$.
Note that for $ |\alpha|\leq N$,
\begin{equation}\label{L-high-E}
\square \partial_x^\alpha u^{(m)}=\partial_x^\alpha f_\delta(u^{(m-1)},\p u^{(m-1)})
\end{equation}
and
\begin{equation}\label{L-05}
\begin{aligned}
&\|\partial_x^\alpha f_\delta(u^{(m-1)},\p u^{(m-1)})\|_{L^2}\\
&\leq
C \bigg(\sup_{1\leq |\beta|\leq {|\alpha|} }|\nabla_{u,\p u}^{\beta}f_\delta(u^{(m-1)},\p u^{(m-1)})|\|(u^{(m-1)},\p u^{(m-1)})\|_{L^\infty}^{|\beta|-1}\bigg)
\sup_{|\beta|=|\alpha| }\|\partial_x^\beta (u^{(m-1)},\p u^{(m-1)})\|_{L^2}\\
&\leq C(\delta^{-1}A_{m-1}^2(t)+\delta^2)A_{m-1}(t)
\end{aligned}
\end{equation}
by standard Sobolev's interpolation inequality.

Then applying the standard energy estimate in Chapter 6 of\cite{H} for the linear wave equation,
one has
\begin{equation} \label{L-06}
\begin{aligned}
\sum\limits_{\scriptstyle\left| \alpha  \right| \le N\hfill\atop
\scriptstyle\left| \beta  \right| = 1\hfill} \| \partial^\beta\partial_x^{\alpha}u^{(m)}(t,x)\|_{L^2}\leq 2 (\sum\limits_{\scriptstyle\left| \alpha  \right| \le N \hfill\atop
\scriptstyle\left| \beta  \right| = 1\hfill}\|\partial^\beta\partial_x^{\alpha}u^{(m)}(2,x)\|_{L^2}
+\int_{2}^{t}\sum_{|\alpha|\leq N}\|\partial_x^\alpha f_\delta(u^{(m-1)},\p u^{(m-1)})\|_{L^2} dt').
\end{aligned}
\end{equation}
Due to $\|u^{(m)}(t,x)\|_{L^2}\leq \|u^{(m)}(2,x)\|_{L^2}+\int_{2}^{t}\|\partial_{t'}u^{(m)}(t',x)\|_{L^2}dt'$,
then it follows from \eqref{L-06} that
\begin{equation}\label{L-07}
\begin{aligned}
\sum_{|\beta|\leq 1}\sum_{|\alpha|\leq N}&\|\partial^\beta\partial_x^{\alpha}u^{(m)}(t,x)\|_{L^2}\leq 2\left({\sum_{|\beta|\leq 1}\sum_{|\alpha|\leq N}\|\partial^\beta\partial_x^{\alpha}u^{(m)}(2,x)\|_{L^2} }\right.
\\
&~~~\left.
{+\int_{2}^{t}\left(\sum_{|\alpha|\leq N}\|\partial_x^\alpha f_\delta(u^{(m-1)},\p u^{(m-1)})\|_{L^2}+\|\partial_{t'}u^{(m)}(t',x)\|_{L^2}\right)dt'}\right).
\end{aligned}
\end{equation}
Substituting \eqref{L-05} into \eqref{L-07} yields that for $t\in [2, T_0]$,
\begin{equation*}
\begin{aligned}
\sum_{|\beta|\leq 1}\sum_{|\alpha|\leq
N}\|\partial^\beta\partial_x^{\alpha}u^{(m)}(t,x)\|_{L^2}
&\leq C\left(A_m(2)
+\int_{2}^{t}
\left(C\left(\delta^{-1}A_{m-1}^2(t')+\delta^2\right)A_{m-1}(t')
+A_m(t')\right)dt'\right)\\
&\leq C\left(A_m(2)+A(T_0-2)\delta^{\nu +1}\max\left\{A^2\delta^{2\nu +1},\delta^2\right\}
+\int_{2}^{t}A_m(t')dt' \right).\\
\end{aligned}
\end{equation*}
By \eqref{L-A} and Gronwall's inequality, we arrive at
\begin{equation*}
\begin{aligned}
A_m(t)&\leq Ce^{C(T_0-2)}\left(\delta^{\nu +1}+A(T_0-2)\delta^{\nu +1}\max\left\{A^2\delta^{2\nu +1},\delta^2\right\}\right)\\
&\leq Ce^{C(T_0-2)}
\left(1+A(T_0-2)\max \left\{A^2\delta^{2\nu +1},\delta^2 \right\}\right)\delta^{\nu +1}.
\end{aligned}
\end{equation*}
Choose the fixed constant $A>0$  such that $A>Ce^{C(T_0-2)}
\left(1+A(T_0-2)\max \left\{A^2\delta^{2\nu +1},\delta^2 \right\}\right)$ for small $\dl$ and $\nu>-\f12$.
Then \eqref{L-A} holds for any $m\in \mathbb{N}\cup \left\{-1\right\}.$

Set
\begin{equation}\label{L-08}
C_m(t):=\|u^{(m)}(t,x)-u^{(m-1)}(t,x)\|_{H^{N+1}}+\|\partial u^{(m)}(t,x)-\partial u^{(m-1)}(t,x)\|_{H^{N}}.
\end{equation}

Note that
\begin{equation*}
\left\{ \begin{array}{l}
\Box(u^{(m)}-u^{(m-1)})=f_\delta(u^{(m-1)},\p u^{(m-1)})-f_\delta(u^{(m-2)},\p u^{(m-2)}),\\
(u^{(m)}-u^{(m-1)})(2,x)=0,~\partial_t(u^{(m)}-u^{(m-1)})(2,x)=0.
\end{array} \right.
\end{equation*}
Then it follows from the standard energy estimate that for $|\alpha|\leq N$,
\begin{equation} \label{L-09}
\begin{aligned}
&\sum_{|\beta|\leq 1}\sum_{|\alpha|\leq N}\|\partial^\beta\partial_x^{\alpha}(u^{(m)}-u^{(m-1)})(t,x)\|_{L^2}
\\
\leq &C\int_{2}^{t}\left(\sum_{|\alpha|\leq N}\|\partial_x^\alpha (f_\delta(u^{(m-1)},\p u^{(m-1)})-f_\delta(u^{(m-2)},\p u^{(m-2)}))\|_{L^2}+\|\partial_{t'}(u^{(m)}-u^{(m-1)})(t',x)\|_{L^2}\right)dt'.
\end{aligned}
\end{equation}
Note that
\begin{equation*}
\begin{aligned}
&\left\|\partial_x^\alpha \left(f_\delta(u^{(m-1)},\p u^{(m-1)})-f_\delta(u^{(m-2)},\p u^{(m-2)})\right)\right\|_{L^2}
\\
\leq& \sum\limits_{j,k,l =  - 1}^3 {\t g}^{jkl}(\dl)
\left( {{{\left\| {\partial _x^\alpha \left( {( {{\partial _j}{u^{(m-1)}} - {\partial _j}{u^{(m-2)}}}){\partial _k}{u^{(m-1)}}{\partial _l}{u^{(m-1)}}} \right)} \right\|}_{{L^2}}}
 } \right.\\
&\left.{  + {{\left\| {\partial _x^\alpha \left( {( {{\partial _k}{u^{(m-1)}} - {\partial _k}{u^{(m-2)}}}){\partial _j}{u^{(m-2)}}{\partial _l}{u^{(m-1)}}} \right)} \right\|}_{{L^2}}}+ {{\left\| {\partial _x^\alpha \left( {( {{\partial _l}{u^{(m-1)}}
- {\partial _l}{u^{(m-2)}}}){\partial _j}{u^{(m-2)}}{\partial _k}{u^{(m-2)}}} \right)} \right\|}_{{L^2}}}} \right)
\\
&+\dl^2\left\|\partial _x^\alpha\left(u^{(m-1)}-u^{(m-2)}\right)\right\|_{L^2}
\\
\leq &C\left(A^2\delta^{2\nu+1}+\dl^2\right)C_{m-1}(t),
\end{aligned}
\end{equation*}
where we have used standard Sobolev's interpolation inequality.
Then by \eqref{L-09}, one has that for $t\in [2, T_0]$,
\begin{equation*}
\begin{aligned}
&\sum_{|\beta|\leq 1}\sum_{|\alpha|\leq N}\|\partial^\beta\partial_x^{\alpha}(u^{(m)}-u^{(m-1)})(t,x)\|_{L^2}
\leq &C\int_{2}^{t}\left(\left(A^2\delta^{2\nu+1}+\dl^2\right)C_{m-1}(t')+C_m(t')\right)dt'.
\end{aligned}
\end{equation*}
Due to the definition \eqref{L-08} and Gronwall's lemma, then
\begin{equation*}
\begin{aligned}
C_m(t)&\leq  C\int_{2}^{t}\left(\left(A^2\delta^{2\nu+1}+\dl^2\right)C_{m-1}(t')+C_m(t')\right)dt'
\\
&\leq C \max \left\{ A^2\delta^{2\nu+1},\dl^2\right\}\int_{2}^{t} C_{m-1}(t')dt'\exp\left\{\int_{2}^{t}Cdt'\right\}
\\
&\leq C_{T_0,A,\delta} \int_{2}^{t} C_{m-1}(t')dt',
\end{aligned}
\end{equation*}
where $C_{T_0,A,\delta}=Ce^{C(T_0-2)}\max \left\{A^2\delta^{2\nu+1},\dl^2\right\}$. Thus, for $t\in [2, T_0]$,
 \begin{equation*}
\begin{aligned}
C_m(t)\leq C_{T_0,A,\delta}^{m-1}\int_{2}^{t}\cdots \int_{2}^{\tau_3}C_1(\tau_2)d\tau_2d\tau_3\cdots d\tau_m
\leq\frac{(C_{T_0,A,\delta} (T_0-2))^{m-1}}{ (m-1)!}\sup_{2\leq t\leq T_0}C_1(t).
\end{aligned}
\end{equation*}
This means that there exists $u\in C([2,T_0],H^{N+1})\cap C^1([2,T_0],H^{N})$ such that
$u^{(m)}\to u$ in $C([2,T_0],H^{N+1})\cap C^1([2,T_0],H^{N})$ and $u$ is a solution to problem \eqref{scaling equation}.
Moreover, $u$ satisfies \eqref{L-S} due to \eqref{L-A}.
\end{proof}

Based on the local existence in Proposition 3.1, when $s_0^*\le s=\sqrt{t^2-|x|^2} \le B$ and $B$ is a fixed constant with $s_0<B
\le\sqrt{2T_0+1}$,
we next give the estimate for the energy  $M_N^\delta(s)$ in \eqref{2-15}.

\begin{prop}
For $s_0^*\le s \le B$, it holds that for small $\dl$,
\begin{equation}
  M_N^\delta(s)=\sum_{|I|+|J|\leq N}\mathcal{E}^\delta(s,\partial^IL^Ju)^{\frac{1}{2}}\leq
  C{\delta ^{\nu  +1}}.
\end{equation}\label{L-10}
\end{prop}
\begin{proof}
Set $U_{IJ}=\p^IL^Ju$.
Acting  $\p^IL^J$ and $2\partial_tU_{IJ}$ on both sides of   \eqref{scaling equation} respectively,
then it follows from the integration on the domain $V_s$ that
\begin{equation}\label{L-11}
\mathcal{E}^\delta(s,U_{IJ})=\int_{D_s}\left(|\p U_{IJ}|^2+\dl^2U_{IJ}^2\right)dx
+\int_{V_s}\left(2\p_t U_{IJ}\p^IL^JF_\dl\left(u,\p u\right)\right)dtdx.
\end{equation}

Note that
\begin{equation*}
\int_{D_s}\left(|\p U_{IJ}|^2+\dl^2U_{IJ}^2\right)dx=\int_{\sqrt{4-s^2}}^1\int_{\Bbb S^2} \rho^2\left(|\p U_{IJ}|^2+\dl^2U_{IJ}^2\right)(\rho,\omega)d\omega d\rho,
\end{equation*}
where $d\omega$ stands for the area element on unit sphere $\Bbb S^2$.
A direct computation yields
\begin{equation*}
\begin{aligned}
\frac{d}{ds}\int_{D_s}\left(|\p U_{IJ}|^2+\dl^2U_{IJ}^2\right)dx&=\frac{s}{\sqrt{4-s^2}}\int_{\Bbb S^2} (4-s^2)\left(|\p U_{IJ}|^2+\dl^2U_{IJ}^2\right)\left(\sqrt{4-s^2},\omega\right)d\omega
\\
&=\frac{s}{\sqrt{4-s^2}}\int_{\Lambda_s}\left(|\p U_{IJ}|^2+\dl^2U_{IJ}^2\right)d\sigma ,
\end{aligned}
\end{equation*}
where $d\sigma=(4-s^2)d\omega$ represents the area element of $\Lambda_s$.

Let $H$ be the Heaviside function and $\delta_0$ be the Dirac function. Then one has
\begin{equation*}
\begin{aligned}
  \frac{\partial }{{\partial s}}H( {\sqrt {{s^2} + {{\left| x \right|}^2}}  - t}) = \frac{s}{{\sqrt {{s^2}
  + {{\left| x \right|}^2}} }}{\delta _0}\left( {\sqrt {{s^2} + {{\left| x \right|}^2}}  - t} \right)
\end{aligned}
\end{equation*}
and
\begin{equation}\label{CY-10}
\begin{aligned}
\frac{d}{ds}\int_{V_s}\left(2\p_t U_{IJ}\p^IL^JF_\dl\left(u,\p u\right)\right)dtdx&=\int_{\mathbb{R}^{1+3}}
 \frac{{\partial H( {\sqrt {{s^2} + {{\left| x \right|}^2}}-t} )}}{{\partial s}} \left(2{\partial _t}U_{IJ}\p^IL^JF_\delta (u,\partial u)\right)dtdx   \\
&=\int_{\varkappa_s}\left(2\frac{s}{t}\p_t U_{IJ}\p^IL^JF_\dl\left(u,\p u\right)\right)dx.
\end{aligned}
\end{equation}
Therefore,
\begin{equation}\label{L-12}
\begin{aligned}
\frac{d}{ds}\mathcal{E}^\delta(s,U_{IJ})&=\frac{s}{\sqrt{4-s^2}}\int_{\Lambda_s}\left(|\p U_{IJ}|^2+\dl^2U_{IJ}^2\right)d\sigma
+\int_{\varkappa_{s}} \left(2\frac{s}{t}\p_t U_{IJ}\p^IL^JF_\dl\left(u,\p u\right)\right)dx
\\
&\le\frac{s}{\sqrt{4-s^2}}\int_{\Lambda_s}\left(|\p U_{IJ}|^2+\dl^2U_{IJ}^2\right)d\sigma+\int_{\varkappa_s}\left|\frac{s}{t}\p_t U_{IJ}\right|^2dx+\int_{\varkappa_s}\left|\p^IL^JF_\dl(u,\p u)\right|^2dx.
\end{aligned}
\end{equation}
By Lemma 2.2, we have
\begin{equation*}
\begin{aligned}
\p^IL^JF_\dl(u,\p u)&= \sum\limits_{\scriptstyle I_1+I_2+I_3=I\hfill\atop
\scriptstyle J_1+J_2+J_3=J\hfill}\sum_{j,k,l=-1}^3 \t{g}^{jkl}(\dl) \p^{I_1}L^{J_1}\p_ju\p^{I_2}L^{J_2}\p_ku\p^{I_3}L^{J_3}\p_lu\\
&= \sum\limits_{\scriptstyle I_1+I_2+I_3=I\hfill\atop
\scriptstyle J_1+J_2+J_3=J\hfill}\sum\limits_{\scriptstyle |J_1^'|\le |J_1|\hfill\atop
{\scriptstyle |J_2^'|\le |J_2|\hfill\atop
\scriptstyle |J_3^'|\le |J_3|\hfill}} \sum\limits_{\scriptstyle j,k,l=-1\hfill\atop
\scriptstyle j',k',l'=-1\hfill}^{3}     \t{g}^{jkl}(\dl) c_{jj' }^{J_1^'J_1} c_{kk'}^{J_2^'J_2}
c_{ll'}^{J_3^'J_3}\p_{j'} \p^{I_1}L^{J_1^'} u \cdot \p_{k'}\p^{I_2}L^{J_2^'} u \cdot\p_{l'}\p^{I_3}L^{J_3^'}u.
\end{aligned}
\end{equation*}
Note that at most one  case holds among $|I_1|+|J_1|>\frac{N}{2}$, $|I_2|+|J_2|>\frac{N}{2}$ and $|I_3|+|J_3|>\frac{N}{2}$.
Without loss of generality, $|I_1|+|J_1|+|I_2|+|J_2|\le \frac{N}{2}$ is assumed. Due to $N\ge4$, one can
estimate $\left|\p_{j'} \p^{I_1}L^{J_1^'} u \cdot \p_{k'}\p^{I_2}L^{J_2^'} u\right|$ by \eqref{L-S}. Together with
$\frac{\sqrt{3}}{T_0}\le \frac{s}{t}\le 1$ for  $(t,x)\in [2,T_0]\times \Bbb R^3$, this yields
\begin{equation*}
\int_{\varkappa_s}\left|\p^IL^JF_\dl(u,\p u)\right|^2dx\le C\dl^{4\nu+2}\sum_{|I_3|+|J_3^'|\le N}\mathcal{E}^\delta(s,U_{I_3J_3^'}).
\end{equation*}
By integrating \eqref{L-12} from $s_0^*$ to $s$, one arrives at
\begin{equation}\label{CY-11}
\begin{aligned}
\mathcal{E}^\delta(s,U_{IJ})\le \int_{D_s}\left(|\p U_{IJ}|^2+\dl^2U_{IJ}^2\right)dx
+\int_{s_0^*}^s\left(1+C\dl^{4\nu+2}\right)  \sum_{|I'|+|J'|\leq N}\mathcal{E}^\delta(s',U_{I'J'})ds'.
\end{aligned}
\end{equation}
Therefore, summing \eqref{CY-11} over $|I|+|J|\le N$ and applying Gronwall's lemma  deduce
\begin{equation}\label{CY-12}
\begin{aligned}
\sum_{|I|+|J|\leq N}\mathcal{E}^\delta(s,U_{IJ})&\le \int_{D_s}\sum_{|I|+|J|\le N}\left(|\p U_{IJ}|^2+\dl^2U_{IJ}^2\right)dx+C\int_{s_0^*}^s\left(1+C\dl^{4\nu+2}\right) \sum_{|I|+|J|\leq N}\mathcal{E}^\delta(s',U_{IJ})ds'
\\
&\le e^{2C\left(1+C\dl^{4\nu+2}\right) }\int_{\left\{|x|\le 1\right\}}\sum_{|I|+|J|\le N}\left(|\p U_{IJ}|^2+\dl^2U_{IJ}^2\right)(2,x)dx.
\end{aligned}
\end{equation}
Due to $\nu>-\frac{1}{2}$, then one has from \eqref{CY-12} that for small $\delta$
and $s_0^*\le s \le B$,
\begin{equation*}
M_N^\delta(s)\le C\dl^{\nu +1}.
\end{equation*}
\end{proof}

\section{\textbf{Proof of Theorem \ref{mainthm}}}

By Proposition 3.2, we have
\begin{equation}\label{G-01}
M_N^\delta(s_0)=\sum_{|I|+|J|\leq N}\mathcal{E}^\delta(s_0,\partial^IL^Ju)^{\frac{1}{2}}\leq
C_0{\delta ^{\nu  +1}},
\end{equation}
where $C_0>0$ is a fixed constant derived by Proposition 3.2.
Assume that for $Q=3C_0$, any fixed $\tilde{T}>2$ and
small $\dl>0$, it holds
\begin{equation}\label{G-02}
M_N^\delta(s)\leq Q{\delta ^{\nu+ 1}}\quad\text{as $2\le s \le \tilde{T}$}.
\end{equation}
Next we show  that  \eqref{G-02} holds for $Q/2$ instead of $Q$. Acting   $2\partial_tu$ on both sides
of   \eqref{scaling equation} and integrating on the domain $G_s$ yield
\begin{equation*}
\begin{aligned}
{\mathcal{E}^\delta }\left( {s,u} \right) = {\mathcal{E}^\delta }\left( {s_0,u} \right)
+\int_{G_s} {2{\partial _t}uF_\delta (u,\partial u)dtdx}.
\end{aligned}
\end{equation*}

\begin{lem}
It holds that
\begin{equation}\label{G-03}
\begin{aligned}
 {{\mathcal{E}}^\delta(s,u)}^{1/2}\leq
 {{\mathcal{E}}^\delta(2,u)}^{1/2}+\int_2^s \left\| F_\dl(u,\partial u) \right\|_{L_{\varkappa_{s'}}^2}ds'.
\end{aligned}
\end{equation}
\end{lem}
\begin{proof}
As in \eqref{CY-10}, one has
\begin{equation}\label{G-04}
\begin{aligned}
\frac{{\partial {{\mathcal{E}}}^\delta\left( {s,u} \right)}}{{\partial s}}&=\int_{\varkappa_s}
2\frac{s}{t}{\partial _t}uF_\delta (u,\partial u)dx.
\end{aligned}
\end{equation}
Note that
\begin{equation*}
\begin{aligned}
\int_{\varkappa_s} |2\frac{s}{t}{\partial _t}uF_\delta (u,\partial u)|dx \le
2{\left\| F_\dl(u,\partial u) \right\|_{L_{\varkappa_s}^2}}{\left\| {\frac{s}{t}{\partial _t}u} \right\|_{L_{\varkappa_s}^2}}
\le  2{\left\| F_\dl(u,\partial u) \right\|_{L_{\varkappa_s}^2}}{{\mathcal{E}}^\delta(s,u)}^{1/2}.
\end{aligned}
\end{equation*}
This, together with \eqref{G-04}, yields
\begin{equation*}\label{4-23}
\begin{aligned}
\frac{{\partial  {{\mathcal{E}}}^\delta\left( {s,u} \right)}}{{\partial s}}&\leq 2{\left\|F_\dl(u,\partial u)\right\|_{L_{\varkappa_s}^2}}{{\mathcal{E}}^\delta(s,u)}^{1/2}.
\end{aligned}
\end{equation*}
Namely,
\begin{equation*}
\begin{aligned}
  \frac{{\partial { {\mathcal{E}}^\delta\left( {s,u} \right)}}^{1/2}}{{\partial s}}\leq \left\| F_\dl(u,\partial u) \right\|_{L_{\varkappa_s}^2}.
\end{aligned}
\end{equation*}
Therefore,
\begin{equation*}
\begin{aligned}
{{\mathcal{E}}^\delta(s,u)}^{1/2}
\leq{\mathcal{E}^\delta(2,u)}^{1/2}+\int_2^s \left\| F_\dl(u,\partial u) \right\|_{L_{\varkappa_{s'}}^2}ds'.
\end{aligned}
\end{equation*}
\end{proof}

\begin{prop}
If $M_N^\delta(s)\leq Q{\delta ^{\nu  + 1}}$ holds for $s\in [s_0,\tilde{T}]$, we then have that for small $\dl$,
\begin{equation}\label{G-05}
M_N^\delta(s)\leq \frac{Q}{2}{\delta ^{\nu+1}}.
\end{equation}
\end{prop}
\begin{proof}
It follows from a direct computation that for $|I|\leq N$,
\begin{equation*}
\begin{aligned}
\left( \Box + {\delta ^2} \right){L^I}u =  {L^I}{F_\delta }\left( {u,\partial u} \right).
\end{aligned}
\end{equation*}
Applying Lemma 4.1 yields
\begin{equation*}
\begin{aligned}
 {{\mathcal{E}}^\delta(s, {{L^I}u})}^{1/2}\leq
 {{\mathcal{E}}^\delta(2, {{L^I}u})}^{1/2}+\int_2^s \left\|{L^I}{F_\delta }\left( {u,\partial u} \right) \right\|_{L_{\varkappa_{s'}}^2}ds'.
\end{aligned}
\end{equation*}
By Lemma 2.2, one has
\begin{equation}\label{G-06}
\begin{aligned}
L^IF_\delta \left( {u,\partial u} \right) = \sum\limits_{{I_1} + {I_2}+I_3 = I} \sum\limits_{\scriptstyle |I_1^'|\le |I_1|\hfill\atop
{\scriptstyle |I_2^'|\le |I_2|\hfill\atop
\scriptstyle |I_3^'|\le |I_3|\hfill}} \sum\limits_{\scriptstyle j,k,l=-1\hfill\atop
\scriptstyle j',k',l'=-1\hfill}^{3}     \t{g}^{jkl}(\dl) c_{jj'}^{I_1^'I_1} c_{kk'}^{I_2^'I_2} c_{ll'}^{I_3^'I_3}\p_{j'} L^{I_1^'} u \cdot \p_{k'} L^{I_2^'} u\cdot \p_{l'} L^{I_3^'}u.
\end{aligned}
\end{equation}
Note that for $|I|\leq N$,
\begin{equation} \label{G-07}
\begin{aligned}
 \t M_{N}^\delta(s)\leq
\t M_{N}^\delta(2)+\int_2^s\sum\limits_{|I|\leq N} \left\| {L^I}{F_\delta }\left( {u,\partial u} \right)  \right\|_{L_{\varkappa_{s'}}^2}ds',
\end{aligned}
\end{equation}
where $\t M_{N}^\delta(s):=\sum_{|I|\le N}{\mathcal{E}}^\delta(s, {{L^I}u})^{1/2}$.
From \eqref{G-06}, there will appear such typical terms in it:
\begin{equation*}
\begin{aligned}
\left\{ \begin{array}{l}
\delta^{-1}{\partial _\alpha }{L^{I_1^'}}u{\partial _\beta }{L^{I_2^'}}u{\partial _\gamma }{L^{I_3^'}}u,
\quad~~\alpha ,\beta ,\gamma  \in \left\{ {0,1,2,3} \right\},\\
{\partial _\alpha }{L^{I_1^'}}u{\partial _\beta }{L^{I_2^'}}u{L^{I_3^'}}u,\qquad \quad~~~~~  \alpha ,\beta  \in \left\{ {0,1,2,3} \right\},\\
\delta{\partial _\alpha }{L^{I_1^'}}u  {L^{I_2^'}}u {L^{I_3^'}}u,
\qquad \qquad~~~~~  \alpha  \in \left\{ {0,1,2,3} \right\},\\
\delta^2{L^{I_1^'}}u{L^{I_2^'}}u{L^{I_3^'}}u.
\end{array} \right.
\end{aligned}
\end{equation*}
By Lemma 2.3 and Lemma 2.4, when $|I_1|+|I_2|\le  \frac{N}{2}$ is assumed as in Proposition 3.2, one has
\begin{equation*}
\begin{aligned}
\delta^{-1}{\left\| {\partial _\alpha }{L^{I_1^'}}u{\partial _\beta }{L^{I_2^'}}u{\partial _\gamma }{L^{I_3^'}}u\right\|_{{L_{{\varkappa_s}}^2}}} &\le \delta^{-1}s{\left\| {\frac{s}{t}{\partial _\alpha }{L^{I_1^'}}u} \right\|_{{L_{{\varkappa_s}}^\infty }}}{\left\| {\frac{t}{s}{\partial _\beta }{L^{I_2^'}}u} \right\|_{{L_{{\varkappa_s}}^\infty }}}{\left\| {\frac{s}{t}{\partial _\gamma }{L^{I_3^'}}u} \right\|_{{L_{{\varkappa_s}}^2}  }}
\leq CQ^2\delta^{2\nu+1}s^{-3/2}  \t M_{N}^\delta(s),
\\
{\left\|  {\partial _\alpha }{L^{I_1^'}}u{\partial _\beta }{L^{I_2^'}}u {L^{I_3^'}}u\right\|_{{L_{{\varkappa_s}}^2}}} &\le  {\left\| {\frac{s}{t}{\partial _\alpha }{L^{I_1^'}}u} \right\|_{{L_{{\varkappa_s}}^\infty }}}{\left\| {\frac{t}{s}{\partial _\beta }{L^{I_2^'}}u} \right\|_{{L_{{\varkappa_s}}^\infty }}}{\left\| {{L^{I_3^'}}u} \right\|_{{L_{{\varkappa_s}}^2}}}
\leq CQ^2\delta^{2\nu+1}s^{-5/2}  \t M_{N}^\delta(s),
\\
\delta{\left\|  {\partial _\alpha }{L^{I_1^'}}u {L^{I_2^'}}u {L^{I_3^'}}u\right\|_{{L_{{\varkappa_s}}^2}}} &\le \delta s{\left\| {\frac{s}{t}{\partial _\alpha }{L^{I_1^'}}u} \right\|_{{L_{{\varkappa_s}}^\infty }}}{\left\| { { L^{I_2^'}}u} \right\|_{{L_{{\varkappa_s}}^\infty }}}{\left\| {  {L^{I_3^'}}u} \right\|_{{L_{{\varkappa_s}}^2}}}
\leq CQ^2\delta^{2\nu+1}s^{-2} \t M_{N}^\delta(s),
\\
\delta^2{\left\|  {L^{I_1^'}}u {L^{I_2^'}}u {L^{I_3^'}}u\right\|_{{L_{{\varkappa_s}}^2}}} &\le \delta^2{\left\| { {L^{I_1^'}}u} \right\|_{{L_{{\varkappa_s}}^\infty }}}{\left\| { { L^{I_2^'}}u} \right\|_{{L_{{\varkappa_s}}^\infty }}}{\left\| {  {L^{I_3^'}}u} \right\|_{{L_{{\varkappa_s}}^2}}}
\leq CQ^2\delta^{2\nu+1}s^{-3} \t M_{N}^\delta(s).
\end{aligned}
\end{equation*}
Together with \eqref{G-07} and Gronwall's inequality, this yields
\begin{equation*}
\begin{aligned}
\t M_{N}^\delta(s)&\leq
\t M_{N}^\delta(2)+CQ^2\delta^{2\nu+1}\int_2^s {\left(s'\right)}^{-3/2}\t M_{N}^\delta(s')ds'
\end{aligned}
\end{equation*}
and
\begin{equation} \label{G-08}
\begin{aligned}
\t M_{N}^\delta(s)& \le   \t M_{N}^\delta(2)\exp\left\{\int_2^sCQ^2\delta^{2\nu+1}{\left(s'\right)}^{-3/2}ds'\right\}
\\
&\le \t M_{N}^\delta(2)\exp\left\{CQ^2\delta^{2\nu+1}\right\}
\\
&\le \frac{Q}{2}\delta^{\nu +1}\qquad \text{(due to $\nu>-\f12$ and the smallness of $\delta$)}.
\end{aligned}
\end{equation}
Similarly,
\begin{equation}\label{G-09}
 M_N^\delta(s)\le \frac{Q}{2}\delta^{\nu +1}.
\end{equation}
\end{proof}

Next we derive the estimates for the solution in \eqref{0-S}. In the coordinate $(\tau,z)$, without confusions, we still denote $\p=\left(\p_{\tau},\p_{z_1},\p_{z_2},\p_{z_3}\right)$, $L=\left(L_1,L_2,L_3\right)$ and  $L_a=z_a\p_{\tau}+\tau\p_{z_a} (a=1,2,3)$.
Note that the solution $w(\tau,z)$ of \eqref{scaling equation} satisfies
\begin{equation*}
|\partial^IL^Jw(\tau,z)|\ls \left\{ \begin{array}{l}
{{\delta ^{\nu+1} }  },  ~~2\le \tau \le 5/2,\\
{{\delta ^\nu } {{\tau}^{ - \frac{3}{2}}}},~~\tau>5/2
\end{array} \right.
\end{equation*}
by \eqref{L-S}, \eqref{Y-Sob} and \eqref{G-09}.
In addition, since the transformation $(\tau,z)=(\frac{t}{\delta }+2,
\frac{x}{\delta})$ implies $\partial_tu =\delta^{-1}\partial_\tau w$ and $\partial_{x_1}u =\delta^{-1}\partial_{z_1}w $,
 then
\begin{equation*}
\left| {{\partial_\delta ^{\rm{I}}}{L_\delta^J}u\left( {t,x} \right)} \right| \ls  \left\{ \begin{array}{l}
{{\delta ^{\nu+1} }  }, \qquad \quad\quad ~~t\le \delta/2,
\\
{{\delta ^\nu } {{\left(\delta^{-1} t+2\right)}^{ - \frac{3}{2}}}},~~t>\delta/2.
\end{array} \right.
\end{equation*}
Similarly, when $2\le \tau \le 5/2$, according to \eqref{L-S}, $|\partial \partial^IL^Jw(\tau,z)|\ls \delta^{\nu+1}$ holds.
On the other hand,  when $\tau>5/2$, we have  $|\partial \partial^IL^Jw(\tau,z)|
=\frac{\tau}{s}|\frac{s}{\tau}\partial \partial^IL^Jw(\tau,z)|\ls \frac{\tau^{-1/2}}{s}\delta^{\nu+1}\ls \tau^{-1}\delta^{\nu+1}$
by $s=\sqrt{\tau^2-|z|^2}$, \eqref{Y-Sob} and \eqref{G-09}.
Thus,
\begin{equation*}
\left| {\partial_\delta{\partial_\delta ^{\rm{I}}}{L_\delta^J}u\left( {t,x} \right)} \right| \ls  \left\{ \begin{array}{l}
{{\delta ^{\nu+1} }  },\qquad \qquad \quad ~~t\le \delta/2,
\\
{{\delta ^{\nu+1} } {{\left(\delta^{-1} t+2\right)}^{ - 1}}},~~t>\delta/2.
\end{array} \right.
\end{equation*}

\section{\textbf{Proof  of Theorem \ref{mainthm-1}}}

Let $\alpha(\theta)$ be a  Friedrichs mollifier and $\alpha_{\frac{1}{16}}(\theta)=16^3\alpha(16\theta)$.
Set $f(\theta)=(\overline f*\alpha_{\frac{1}{16}})(\theta)\in C_0^\infty(-1,1)$ and $h(\theta)=(\overline h*\alpha_{\frac{1}{16}})(\theta)\in C_0^\infty(-1,1)$, where
 \begin{equation*}
\overline f \left( \theta  \right) = \left\{ \begin{array}{l}
1,\theta  \in \left[ { - \frac{{13}}{{16}}, - \frac{3}{{16}}} \right],\\
0,\theta  \in {\left[ { - \frac{{13}}{{16}}, - \frac{3}{{16}}} \right]^c}
\end{array} \right.\quad\text{and}\quad
\overline h \left( \theta  \right) = \left\{ \begin{array}{l}
1,\theta  \in \left[ { - \frac{{13}}{{16}},\frac{{13}}{{16}}} \right],\\
0,\theta  \in {\left[ { - \frac{{13}}{{16}},\frac{13}{{16}}} \right]^c}.
\end{array} \right.
\end{equation*}
In addition, define $g=(\overline g*\alpha_{\frac{1}{16}})(\theta)\in C_0^\infty(-1,1)$ with
\begin{equation*}\label{5-2}
\overline g \left( \theta  \right) = \left\{ \begin{array}{l}
\theta  + 1,\theta  \in \left[ { - \frac{13}{16},\frac{9}{16}} \right],\\
0,~~~~~~~\theta  \in {\left[ { - \frac{13}{16},\frac{9}{16}} \right]^c}.
\end{array} \right.
\end{equation*}
It is easy to know that
\begin{equation*}\label{5-1}
\left\{ \begin{array}{l}
f\left( \theta\right)  =1,~~\theta \in \left[ { - \frac{3}{4}, - \frac{1}{4}} \right],\\
f\left( \theta \right) \ge0,~~\theta \in \left[ { - \frac{3}{4}, - \frac{1}{4}} \right]^c,
\end{array} \right.
~~
\left\{ \begin{array}{l}
h\left( \theta \right) =1,~~\theta \in \left[ { - \frac{3}{4},  \frac{3}{4}} \right],\\
h\left( \theta \right) \ge 0,~~\theta \in\left[ { - \frac{3}{4},  \frac{3}{4}} \right]^c
\end{array} \right.
\end{equation*}
and
\begin{equation*}\label{5-2}
\left\{ \begin{array}{l}
g'\left( \theta \right)=1,~~\theta \in \left[ { - \frac{3}{4}, - \frac{1}{4}} \right],\\
g'\left( \theta \right) \ge 0,~~\theta \in \left[ { - 1, 0} \right].
\end{array} \right.
\end{equation*}
Choose ${u_0}\left( x \right) = 4g\left( {{x_1}} \right)h\left( {{x_2}} \right)h\left( {{x_3}} \right)$
and ${  u_1}\left( x \right)= 4f\left( {{x_1}} \right)h\left( {{x_2}} \right)h\left( {{x_3}} \right)$ in \eqref{equation}.
Define
\begin{equation*}\label{5-5}
\Upsilon(t)=\int_{{\mathbb{R}^3}} {{\partial _t}u{\partial _{{1}}}udx}.
\end{equation*}
It follows from a direct computation that for $\nu \le -\frac{1}{2}$,
\begin{equation*}
\begin{aligned}
\Upsilon(0)=\int_{{\mathbb{R}^3}} {{\delta^{2\nu+1}{u_1}\left( {\frac{x}{\delta }} \right)\partial _{{1}}}{u_0}\left( {\frac{x}{\delta }} \right)dx}  &> \int_{\left[ { - \frac{3}{4}\dl, - \frac{1}{4}\dl} \right] \times {{\left[ { - \frac{3}{4}\dl,  \frac{3}{4}\dl} \right]}^2}} {{\delta^{2\nu+1}{u_1}\left( {\frac{x}{\delta }} \right)\partial _{{1}}}{u_0}\left( {\frac{x}{\delta }} \right)dx}
\\
&\ge \frac{1}{2} \cdot {\left( {\frac{3}{2}} \right)^2}{4^2}{\delta ^3}{\delta ^{ - 1}}
= 18{\delta ^2}.
\end{aligned}
\end{equation*}
Assume that  there exists a classical solution $u\in C^2\left(\left[0,T^*\right)\times \mathbb{R}^3\right)$
to $\square u+u=(\partial_t u)^2\partial_1u$ with $T^*>0$ being the maximal existence time.
Due to
\begin{equation*}
\begin{aligned}
\Upsilon'(t)= \int_{{\mathbb{R}^3}} {\left( {{\partial _{t}^2}u{\partial _{{1}}}u + {\partial _t}u{\partial _{t1}^2}u} \right)dx}
= \int_{{\mathbb{R}^3}} {\left( {\Delta u - u + {{\left( {{\partial _t}u} \right)}^2}{\partial _{{1}}}u} \right){\partial _{{1}}}udx}+ \int_{{\mathbb{R}^3}} {{\partial _t}u{\partial _{t{1}}^2}udx}
\end{aligned}
\end{equation*}
and
\begin{equation*}
\begin{aligned}
\int_{{\mathbb{R}^3}} {\left( {\Delta u - u} \right){\partial _{{1}}}udx}
= \int_{{\mathbb{R}^3}} {{\partial _t}u{\partial _{t{1}}^2}udx}= 0,
\end{aligned}
\end{equation*}
then
\begin{equation*}
\begin{aligned}
\Upsilon'\left( t \right) = \int_{{\mathbb{R}^3}} {{{\left( {{\partial _t}u{\partial _{{1}}}u} \right)}^2}dx}.
\end{aligned}
\end{equation*}
On the other hand,
\begin{equation*}
\begin{aligned}
{\Upsilon^2}\left( t \right) &\le {\left\| {{\partial _t}u{\partial _{{1}}}u} \right\|_{{L^1}\left( {{\mathbb{R}^3}} \right)}}{\left\| {{\partial _t}u{\partial _{{1}}}u} \right\|_{{L^1}\left( {{\mathbb{R}^3}} \right)}}
\\
&\le {\left( {{{\left( {t + \delta } \right)}^{{3 \mathord{\left/
{\vphantom {3 2}} \right.
\kern-\nulldelimiterspace} 2}}}{{\left\| {{\partial _t}u{\partial _{{1}}}u} \right\|}_{{L^2}\left( {{\mathbb{R}^3}} \right)}}} \right)^2}
\\
&\le {\left( {t + \delta } \right)^3}\left\| {{\partial _t}u{\partial _{{1}}}u} \right\|_{{L^2}\left( {{\mathbb{R}^3}} \right)}^2
\\
&=\left( {t + \delta } \right)^3\Upsilon'(t).
\end{aligned}
\end{equation*}
This derives
\begin{equation*}
- \frac{d}{{dt}}\left( {\frac{1}{{\Upsilon\left( t \right)}}} \right) = \frac{{\Upsilon'\left( t \right)}}{{{\Upsilon^2}\left( t \right)}} \ge \frac{1}{{{{\left( {t + \delta } \right)}^3}}}.
\end{equation*}
Thus,
\begin{equation*}
\frac{1}{{\Upsilon\left( 0 \right)}} - \frac{1}{{\Upsilon\left( t \right)}} \ge \frac{1}{2}{\delta ^{ - 2}} -\frac{1}{2} {\left( {t + \delta } \right)^{ - 2}}
\end{equation*}
and
\begin{equation*}
2{\left( {t + \delta } \right)^2} \le \frac{1}{{\frac{1}{2}{\delta ^{ - 2}} - \frac{1}{{\Upsilon\left( 0 \right)}}\left( {1 - \frac{{\Upsilon\left( 0 \right)}}{{\Upsilon\left( t \right)}}} \right)}} < \frac{1}{{\frac{1}{2}{\delta ^{ - 2}} - \frac{1}{{18}}{\delta ^{ - 2}}}}.
\end{equation*}
This means $t < \left( {\frac{3}{{2\sqrt 2 }} - 1} \right)\delta $ and further $T^*\le \left( {\frac{3}{{2\sqrt 2 }} - 1} \right)\delta $.
Consequently, Theorem \ref{mainthm-1} is shown.

\vskip 0.5 true cm
{\bf \color{blue}{Conflict of interest}}
\vskip 0.2 true cm

{\bf On behalf of all authors, the corresponding author states that there is no conflict of interest.}

\vskip 0.3 true cm

{\bf \color{blue}{Data availability}}

\vskip 0.2 true cm

{\bf Data sharing is not applicable to this article as no new data were created.}

\vskip 0.3 true cm

\bibliographystyle{amsplain}

\end{document}